\documentclass[a4paper,11pt]{amsart}

\allowdisplaybreaks

\usepackage{amsmath,amssymb,amsthm}

\DeclareMathOperator{\dom}{Dom}
\DeclareMathOperator{\support}{supp}
\DeclareMathOperator*{\essup}{ess\,sup}

\newtheorem{Th}{Theorem}[section]
\newtheorem{Cor}[Th]{Corollary}

\newtheorem{Lem}{Lemma}[section]
\newtheorem{Rem}[Lem]{Remark}

\newcommand{\lam}{\lambda}
\newcommand{\rnp}{\mathbb{R}^n_{+}}
\newcommand{\rp}{\mathbb{R}_{+}}
\newcommand{\q}{\mathfrak{q}}

\title[Calder\'on-Zygmund operators in the Bessel setting]
	{Calder\'on-Zygmund operators in the Bessel setting}

\author[J.J. Betancor]{J.J. Betancor}
\address{Jorge J. Betancor \newline
		Departamento de An\'alisis Matem\'atico, 
		Universidad de la Laguna \newline
		Campus de Anchieta, Avda. Astrof\'{\i}sico Francisco S\'anchez, s/n, \newline
		38271 La Laguna (Sta. Cruz de Tenerife), Spain
		}
\email{jbetanco@ull.es}

\author[A.J. Castro]{A.J. Castro}
\address{Alejandro J. Castro \newline
		Departamento de An\'alisis Matem\'atico, 
		Universidad de la Laguna \newline
		Campus de Anchieta, Avda. Astrof\'{\i}sico Francisco S\'anchez, s/n, \newline
		38271 La Laguna (Sta. Cruz de Tenerife), Spain
		}
\email{ajcastro@ull.es}

\author[A. Nowak]{A. Nowak}
\address{Adam Nowak, \newline
			Instytut Matematyczny,
      Polska Akademia Nauk, \newline
      \'Sniadeckich 8,
      00--956 Warszawa, Poland \newline
			\indent and \newline
			Instytut Matematyki i Informatyki,
      Politechnika Wroc\l{}awska,       \newline
      Wyb{.} Wyspia\'nskiego 27,
      50--370 Wroc\l{}aw, Poland      
      }
\email{adam.nowak@pwr.wroc.pl}

\begin{document}


\begin{abstract}
We study several fundamental operators in harmonic analysis related to Bessel operators,
including maximal operators related to heat and Poisson semigroups,
Littlewood-Paley-Stein square functions, multipliers of Laplace transform type and Riesz transforms.
We show that these are (vector-valued)
Calder\'on-Zygmund operators in the sense of the associated space 
of homogeneous type, and hence their mapping properties follow from the general theory.
\end{abstract}

\thanks{
The first-named author was partially supported by MTM2007/65609. 
The second-named author was supported by a grant for Master studies of ``la Caixa''.
The third-named author was partially supported by MNiSW Grant N N201 417839.
}

\maketitle

\section{Introduction}

In their seminal article \cite{MS} Muckenhoupt and Stein investigated in a systematic way harmonic
analysis associated with ultraspherical expansions and their continuous counterparts, Hankel transforms.
That paper is considered as a starting point of an important development connecting harmonic analysis
and discrete and continuous orthogonal expansions. Later many authors contributed to the subject by
studying various questions in different settings, including in particular expansions into classical
orthogonal polynomials; see \cite[Section 1]{BCC1} for sample references. 
In the recent years one can observe an increasing interest in harmonic analysis
of orthogonal expansions, as confirmed by the attention of many mathematicians and numerous papers.

In this article we study several fundamental harmonic analysis operators in the $n$-dimensional setting
related to Hankel transforms. This framework is connected with the Bessel operator
$$
\Delta_{\lam} = - \Delta - \sum_{i=1}^n \frac{2\lam_i}{x_i}\, \frac{\partial}{\partial x_i},
$$
where $\lam \in [0,\infty)^n$ is a multi-index. The operator $\Delta_{\lam}$ will play in our
considerations a similar role to that of the Euclidean Laplacian in the classical setting. It is formally
self-adjoint and nonnegative in $L^2(\rnp,d\mu_{\lam})$, where $\rnp=(0,\infty)^n$ and 
$$
d\mu_{\lam}(x) = x_1^{2\lam_1} \cdot \ldots \cdot x_n^{2\lam_n} \, dx, \qquad x \in \rnp.
$$
The spectral decomposition of $\Delta_{\lam}$, or rather its suitable self-adjoint extension, is given
via the Hankel transform, see Section \ref{sec:prel} for details.
Moreover, $\Delta_{\lam}$ admits the decomposition $\Delta_{\lam} = \sum_i D^{*}_iD_i$, where
$D_i = \partial\slash \partial x_i$ are the usual partial derivatives, and $D_i^*$ are their formal
adjoints in $L^2(\rnp,d\mu_{\lam})$. Thus $D_i$, $i=1,\ldots,n$, are naturally associated derivatives
with the Bessel operator.

The main objects of our study are the following operators related to $\Delta_{\lam}$
(see Section \ref{sec:prel} for strict definitions):
\begin{itemize}
\item maximal operators
$$
W^{\lam}_{*}\colon f \mapsto \big\|\exp(-t\Delta_{\lam})f\big\|_{L^{\infty}(dt)}, \qquad
P^{\lam}_{*}\colon f \mapsto \big\|\exp\big(-t\sqrt{\Delta_{\lam}}\,\big)f\big\|_{L^{\infty}(dt)},
$$
\item Littlewood-Paley-Stein type square functions
\begin{align*}
& g^{\lam,W}_{m,k,r}\colon f\mapsto \big\| D^m \partial_t^k \exp(-t\Delta_{\lam})f
	\big\|_{L^r(t^{(|m|\slash 2+k)r-1}dt)}, \\
& g^{\lam,P}_{m,k,r}\colon f\mapsto \big\| D^m \partial_t^k \exp\big(-t\sqrt{\Delta_{\lam}}\,
	\big)f \big\|_{L^r(t^{(|m|+k)r-1}dt)},
\end{align*}
\item multipliers of Laplace transform type
\begin{align*}
& T^{\lam}_{\mathcal{M}_W}\colon f \mapsto \Delta_{\lam} \int_0^{\infty} 
	\exp(-t\Delta_{\lam})f\, \psi(t)\, dt, \\
& T^{\lam}_{\mathcal{M}_P}\colon f \mapsto \sqrt{\Delta_{\lam}} \int_0^{\infty} 
	\exp\big(-t\sqrt{\Delta_{\lam}}\, \big)f\, \psi(t)\, dt,
\end{align*}
\item Riesz transforms
$$
R_m^{\lam}\colon f \mapsto D^m (\Delta_{\lam})^{-|m|\slash 2}f.
$$
\end{itemize} 
We treat all these operators in a unified way, by means of the Calder\'on-Zygmund theory.
Our main result, Theorem \ref{thm:main}, says that these are either scalar-valued, or can be viewed as 
vector-valued, Calder\'on-Zygmund operators in the sense of the triple 
$(\rnp,d\mu_{\lam},|\cdot|)$, where $|\cdot|$ stands for the ordinary distance. 
According to the terminology of Coifman and Weiss \cite{CW}, this triple forms a space of homogeneous
type (this means, in particular, that the measure $\mu_{\lam}$ possesses the doubling property).
Consequences are then
delivered by the general theory. In particular, we conclude mapping properties in weighted $L^p$ spaces.

Typically the main difficulty related to the Calder\'on-Zygmund approach is to show suitable kernel
estimates. Inspired by earlier ideas used in certain Laguerre settings \cite{NS2,Sa,Szarek}, we present
a convenient and transparent technique based on an integral representation of the Bessel heat kernel
that emerges from Schl\"afli's Poisson type formula for the modified Bessel function of the first kind,
see Section \ref{sec:prep}. It is remarkable that a similar method has been developed recently
by Nowak and Sj\"ogren \cite{NoSj3} in the more complex setting of classical Jacobi expansions.

Our present results constitute a continuation and extension of many earlier investigations in the Bessel
setting. This concerns the fundamental paper \cite{MS} as well as more recent articles, see for instance
\cite{AK,BBFMT,BFMR,BFMT,BFS,BHNV,BMR,Stem}. In all the mentioned cases, the one dimensional situation
was considered. The study of the multi-dimensional Bessel setting has been undertaken only recently by
Betancor, Castro and Curbelo \cite{BCC1,BCC2}, and by methods not involving the Calder\'on-Zygmund theory.
The results of this paper are also related to those of \cite{BCC1,BCC2}.

The paper is organized as follows. Section \ref{sec:prel} contains the setup, strict definitions of
the investigated operators and statements of the main results. In Section \ref{sec:prep} we gather
various preparatory facts and results needed to furnish the proof of the main theorem, which is then done
in Section \ref{sec:proof}.

Throughout the paper we use a standard notation with essentially all symbols referring to the space
of homogeneous type $(\rnp,d\mu_{\lam},|\cdot|)$. 
Thus $C_c^{\infty}(\rnp)$ denotes the space of smooth and compactly supported functions in $\rnp$.
By $\langle f, g\rangle_{d\mu_{\lam}}$ we mean
$\int_{\rnp} f(x)\overline{g(x)}d\mu_{\lam}(x)$ whenever the integral makes sense. Further,
$L^p(wd\mu_{\lam})$ stands for the weighted $L^p$ space, $w$ being a nonnegative weight on $\rnp$.
Given $1\le p < \infty$, $p'$ is its adjoint exponent, $1\slash p + 1\slash p' =1$. For $1\le p < \infty$,
we denote by $A_p^{\lam}= A_p^{\lam}(\rnp,d\mu_{\lam})$ the Muckenhoupt class of $A_p$ weights related
to the measure $d\mu_{\lam}$. While writing estimates, we will use the notation $X \lesssim Y$ to
indicate that $X \le CY$ with a positive constant $C$ independent of significant quantities. We shall
write $X \simeq Y$ when simultaneously $X \lesssim Y$ and $Y \lesssim X$.

\section{Preliminaries and main results} \label{sec:prel}

Let $\lam \in [0,\infty)^n$. For $z\in \rnp$, consider the functions
$$
\varphi_{z}^{\lam}(x) = \prod_{i=1}^n (x_i z_i)^{-\lam_i+1\slash 2} J_{\lam_i-1\slash 2}(x_i z_i),
	\qquad x \in \rnp,
$$
where $J_{\nu}$ denotes the Bessel function of the first kind and order $\nu$, cf. \cite{Wat}.
It is well known that for each $z \in \rnp$, the function $\varphi_{z}^{\lam}$ is an eigenfunction
of the $n$-dimensional Bessel operator $\Delta_{\lam}$, and the corresponding eigenvalue is
$|z|^2= z_1^2+\ldots + z_n^2$,
$$
\Delta_{\lam} \varphi_{z}^{\lam} = |z|^2 \varphi_{z}^{\lam}, \qquad z \in \rnp.
$$
The $n$-dimensional Hankel transform $h_{\lam}$ defined by 
$$
h_{\lam}f(z) = \int_{\rnp} \varphi_{z}^{\lam}(x) f(x)\, d\mu_{\lam}(x), \qquad z \in \rnp,
$$
plays in the Bessel context a similar role as the Fourier transform in the Euclidean setting.
It is well known that $h_{\lam}$ is an isometry in $L^2(d\mu_{\lam})$ and it coincides there
with its inverse, $h_{\lam}^{-1}=h_{\lam}$. Moreover, for sufficiently regular functions $f$,
say $f \in C_c^{\infty}(\rnp)$, we have
\begin{equation*}
h_{\lam}(\Delta_{\lam}f)(z) = |z|^2 h_{\lambda}(f)(z), \qquad z \in \rnp.
\end{equation*}
Note that in dimension one and for $\lambda = 0$ one recovers here the setting of the cosine
transform on the positive half-line. 

We consider the nonnegative self-adjoint extension of $\Delta_{\lam}$ (still denoted by the same symbol)
defined by
\begin{equation} \label{Bes_spec}
\Delta_{\lam}f = h_{\lam}(|z|^2 h_{\lam}f), \qquad f \in \dom(\Delta_{\lam}),
\end{equation}
on the domain
$$
\dom(\Delta_{\lam}) = \big\{ f\in L^2(d\mu_{\lam}) : |z|^2 h_{\lam}f \in L^2(d\mu_{\lam}) \big\}.
$$
Then the spectral decomposition of $\Delta_{\lam}$ is given via the Hankel transform.

The semigroup $\{W_t^{\lam}\}=\{\exp(-t\Delta_{\lam})\}$ generated by $-\Delta_{\lam}$ is usually
referred to as the heat semigroup associated with the Bessel operator, or simply the Bessel heat semigroup.
It has the integral representation
\begin{equation} \label{int_heat}
W_t^{\lam}f(x) = \int_{\rnp} W_t^{\lam}(x,y)f(y)\, d\mu_{\lam}(y), \qquad f \in L^2(d\mu_{\lam}),
	\qquad x \in \rnp, \quad t>0,
\end{equation}
where the associated heat kernel is given by (see \cite[p.\,395]{Wat})
\begin{align*}
W_t^{\lam}(x,y) & = \int_{\rnp} e^{-t|z|^2} \varphi_z^{\lam}(x) \varphi_z^{\lam}(y)\, d\mu_{\lam}(z)\\
& =	\frac{1}{(2t)^n} \exp\Big( -\frac{1}{4t}\big(|x|^2+|y|^2\big)\Big)
	\prod_{i=1}^n (x_i y_i)^{-\lambda_i+1\slash 2} I_{\lambda_{i}-1/ 2}
		\Big(\frac{x_i y_i}{2t} \Big), \qquad x,y \in \rnp, \;\; t>0;
\end{align*}
here $I_{\nu}$ denotes the modified Bessel function of the first kind and order $\nu$ (cf. \cite{Wat}).
It may be deduced that the integral in \eqref{int_heat} actually converges for more general functions
from weighted $L^p$ spaces with Muckenhoupt weights, producing always smooth functions of 
$(x,t) \in \rnp \times \rp$, and thus providing a good definition of $W_t^{\lam}$ on these spaces,
see Lemma \ref{lem:heat} below.
Moreover, \eqref{int_heat} defines a symmetric diffusion semigroup in 
$L^p(d\mu_{\lam})$, $1\le p \le \infty$, in the sense of Stein \cite[p.\,65]{Stlitt},
which is Markovian; see for instance \cite[Section 6]{NS5}.
The semigroup $\{P_t^{\lam}\} = \{\exp(-t\sqrt{\Delta_{\lam}})\}$ generated by the square root of the
Bessel operator is called the Poisson-Bessel semigroup. By the subordination principle,
it is related to $\{W_t^{\lam}\}$ by
$$
P_t^{\lam}f(x) = \int_0^{\infty} W^{\lam}_{t^2\slash (4u)}f(x) \, \frac{e^{-u}du}{\sqrt{\pi u}},
	\qquad x \in \rnp, \quad t>0.
$$
The maximal operators of these semigroups are defined by
$$
W_{*}^{\lam}f = \sup_{t>0} |W_t^{\lam}f|, \qquad P_{*}^{\lam}f = \sup_{t>0}|P_t^{\lam}f|.
$$
As is well known, mapping properties of $W^{\lam}_{*}$ and $P^{\lam}_{*}$ 
are connected with the boundary behavior of the semigroups. Notice that $P_*^{\lam}f\le W_*^{\lam}f$,
by subordination. 

Littlewood-Paley-Stein square functions based on $\{W_t^{\lam}\}$ and 
$\{P_t^{\lam}\}$ have the general form
\begin{align*}
g_{m,k,r}^{\lam,W}(f)(x) & = \big\| \partial_x^{m}\partial_t^k W_t^{\lam}f(x)
	\big\|_{L^r(t^{(k+|m|\slash 2)r-1}dt)}, \qquad x \in \rnp,\\
g_{m,k,r}^{\lam,P}(f)(x) & = \big\| \partial_x^{m}\partial_t^k P_t^{\lam}f(x)
	\big\|_{L^r(t^{(k+|m|)r-1}dt)}, \qquad x \in \rnp,	
\end{align*}
where $2\le r < \infty$, $k \in \mathbb{N}$, $m \in \mathbb{N}^n$ is a multi-index,
$|m|=m_1+\ldots+m_n$ is its length, and $\partial_x^{m} = \partial_{x_1}^{m_1}\ldots \partial_{x_n}^{m_n}$.
Square functions of this form are important tools in harmonic analysis.

Given a bounded measurable function $\mathcal{M}$ on $\rnp$, the associated Hankel multiplier 
$T^{\lam}_{\mathcal{M}}$ is defined by
$$
T^{\lam}_{\mathcal{M}}f = h_{\lam}(\mathcal{M} h_{\lam}f), \qquad f \in L^2(d\mu_{\lam}).
$$
Clearly, $T^{\lam}_{\mathcal{M}}$ is well defined in $L^2(d\mu_{\lam})$.
We say that $\mathcal{M}$ is of Laplace transform type when it can be represented as
$$
\mathcal{M}(z) =
\mathcal{M}_W(z) = |z|^2\int_0^{\infty} e^{-|z|^2s}\psi(s)\, ds, \qquad z \in \rnp,
$$
or as
$$
\mathcal{M}(z) =
\mathcal{M}_P(z) = |z|\int_0^{\infty} e^{-|z|s}\psi(s)\, ds, \qquad z \in \rnp,
$$
for some bounded function $\psi$ on $\rp$. An important special case here is the choice
$\psi(s)=s^{-i\gamma}\slash \Gamma(1-i\gamma)$, $\gamma$ real, producing the multiplier
$\mathcal{M}_W(z) = |z|^{2i\gamma}$ that corresponds to the imaginary power $\Delta_{\lam}^{i\gamma}$
of the Bessel operator.

Riesz transforms in the Bessel setting are formally given, according to a general concept, by
$$
R_m^{\lam} = \partial^m \Delta_{\lam}^{-|m|\slash 2},
$$
where $m \in \mathbb{N}^n$ and $|m|$ is the order of the transform.
To give this definition a strict meaning, we introduce the space
$$
C^{\lam} = \big\{ f \in C^{\infty}(\rnp) : h_{\lam}f \in C_c^{\infty}(\rnp)\big\}.
$$
This space is a dense linear subspace of $L^2(d\mu_{\lam})$.
We have, see \eqref{Bes_spec},
$$
\Delta_{\lam}^{-|m|\slash 2}f = h_{\lam}(|z|^{-|m|}h_{\lam}f), \qquad f \in C^{\lam}.
$$
If $f \in C^{\lam}$, then $\Delta_{\lam}^{-|m|\slash 2}f \in C^{\infty}(\rnp)$ and therefore
$R_m^{\lam}$ is well defined on $C^{\lam}$. Then $R_m^{\lam}$ can be extended uniquely to a bounded
linear operator on $L^2(d\mu_{\lam})$. All these properties will be justified in detail
in Section \ref{sec:riesz}.

We shall treat all the operators
$$
W_*^{\lam}, \; P_*^{\lam}, \; g^{\lam,W}_{m,k,r}, \; g^{\lam,P}_{m,k,r}, \;  
	T^{\lam}_{\mathcal{M}}, \; R^{\lam}_m,
$$
in a unified way, by means of the general Calder\'on-Zygmund theory.
Clearly, the maximal operators and the $g$-functions are not linear. They are, however, associated with
vector-valued linear operators taking values in some Banach spaces $\mathbb{B}$, where
$\mathbb{B} = L^r(t^{(k+|m|\slash 2)r-1}dt)$ in case of $g_{m,k,r}^{\lam,W}$
and $\mathbb{B} = L^r(t^{(k+|m|)r-1}dt)$ in case of $g_{m,k,r}^{\lam,P}$. 
For $W_*^{\lam}$ and $P_*^{\lam}$ we shall, for technical reasons, choose $\mathbb{B}$ not as 
$L^{\infty}(dt)$ but as the closed and separable subspace $C_0\subset L^{\infty}(dt)$ consisting
of all continuous functions f on $\rp$ which have finite limits as $t\to 0^+$ and vanish as $t\to \infty$.
In all the cases
we shall say that the operator is associated with the corresponding Banach space $\mathbb{B}$.
Similarly, the linear operators $T^{\lam}_{\mathcal{M}}$ and $R^{\lam}_m$
will be said to be associated with the Banach space
$\mathbb{B}=\mathbb{C}$. To obtain mapping properties of our operators, we shall prove that they are
Calder\'on-Zygmund operators in the sense that we now explain.

Let $\mathbb{B}$ be a Banach space and let $K(x,y)$ be a kernel defined on 
$\rnp\times\rnp\backslash \{(x,y):x=y\}$ and taking values in $\mathbb{B}$.
We say that $K(x,y)$ is a standard kernel in the sense of the space of homogeneous type
$(\rnp, d\mu_{\lam},|\cdot|)$ if it satisfies the growth estimate
\begin{equation} \label{gr}
\|K(x,y)\|_{\mathbb{B}} \lesssim \frac{1}{\mu_{\lam}(B(x,|x-y|))}
\end{equation}
and the smoothness estimates
\begin{align}
\| K(x,y)-K(x',y)\|_{\mathbb{B}} & \lesssim \frac{|x-x'|}{|x-y|}\, \frac{1}{\mu_{\lam}(B(x,|x-y|))},
\qquad |x-y|>2|x-x'|, \label{sm1}\\
\| K(x,y)-K(x,y')\|_{\mathbb{B}} & \lesssim \frac{|y-y'|}{|x-y|}\, \frac{1}{\mu_{\lam}(B(x,|x-y|))},
\qquad |x-y|>2|y-y'|; \label{sm2}
\end{align}
here $B(x,R)$ denotes the ball in $\rnp$ centered at $x$ and of radius $R$.
When $K(x,y)$ is scalar-valued, i.e. $\mathbb{B}=\mathbb{C}$, the difference conditions \eqref{sm1}
and \eqref{sm2} are implied by the more convenient gradient condition
\begin{equation} \label{grad}
|\nabla_{\! x,y} K(x,y)| \lesssim \frac{1}{|x-y|\mu_{\lam}(B(x,|x-y|))}.
\end{equation}

Suppose that, for some $1<r<\infty$, $T$ is a linear operator assigning to each $f\in L^r(d\mu_{\lam})$
a measurable $\mathbb{B}$-valued function $Tf$ on $\rnp$. Then $T$ is said to be a (vector-valued)
Calder\'on-Zygmund operator in the sense of the space $(\rnp,d\mu_{\lam},|\cdot|)$ associated with
$\mathbb{B}$ if
\begin{itemize}
\item[(a)] $T$ is bounded from $L^r(d\mu_{\lam})$ to $L^r_{\mathbb{B}}(d\mu_{\lam})$,
\item[(b)] there exists a standard $\mathbb{B}$-valued kernel $K(x,y)$ such that
$$
Tf(x) = \int_{\rnp} K(x,y) f(y)\, d\mu_{\lam}(y), \qquad \textrm{a.e.}\; x \notin \support f,
$$
for every $f \in L_c^{\infty}(\rnp)$,
where $L_c^{\infty}(\rnp)$ is the subspace of $L^{\infty}(\rnp)$ of bounded measurable functions
with compact supports.
\end{itemize}
Here integration of $\mathbb{B}$-valued functions is understood in Bochner's sense.
For the theory of Bochner integrals we refer to \cite{Yos}.
The Bochner-Lebesgue spaces $L^p_{\mathbb{B}}(d\mu_{\lam})$, $1\le p \le \infty$,
are defined to consist of all strongly measurable functions $f\colon \rnp \to \mathbb{B}$ such that
$\|f\|_{L^p_\mathbb{B}(d\mu_{\lam})}<\infty$, where
$$
\|f\|_{L^p_{\mathbb{B}}(d\mu_{\lam})} = 
\begin{cases}
\left( \int_{\rnp} \|f(x)\|^p_{\mathbb{B}}\, d\mu_{\lam}(x)\right)^{1\slash p}, & 1 \le p < \infty \\
\essup_{x\in\rnp} \|f(x)\|_{\mathbb{B}}, & p = \infty
\end{cases}.
$$

According to \cite{CW2} and \cite{MS}, we define the atomic Hardy space 
$H_{\mathbb{B},\textrm{at}}^{1,\lam}$ of $\mathbb{B}$-valued functions on $\rnp$ to consist of
these $f \in L^1_{\mathbb{B}}(d\mu_{\lam})$ which admit the atomic decomposition
$$
f = \sum_{j=1}^{\infty} \alpha_j a_j,
$$
the series being convergent in $L^1_{\mathbb{B}}(d\mu_{\lam})$, where $\alpha_j \in \mathbb{C}$ are
scalars such that $\sum_{j=1}^{\infty} |\alpha_j|< \infty$, and $a_j$ are atoms in the sense
we now describe. A strongly measurable $\mathbb{B}$-valued function $a$ is an atom in the context of
$(\rnp,d\mu_{\lam},|\cdot|)$ if
\begin{itemize}
\item[(A1)] there exist $x_0 \in \rnp$ and $r_0>0$ such that $\support a \subset B(x_0,r_0)$,
\item[(A2)] $\|a\|_{L^{\infty}_{\mathbb{B}}(\rnp)} \le 1\slash \mu_{\lam}(B(x_0,r_0))$,
\item[(A3)] $\int_{\rnp} a(x)\, d\mu_{\lam}(x) = 0$.
\end{itemize}
The norm in $H_{\mathbb{B},\textrm{at}}^{1,\lam}$ is defined by
$$
\|f\|_{H_{\mathbb{B},\textrm{at}}^{1,\lam}} = \inf \sum_{j=1}^{\infty} |\alpha_j|,
$$
where the infimum is taken over all atomic decompositions $f=\sum_{j=1}^{\infty}\alpha_j a_j$ of $f$.
Note that the condition (A2) above may be replaced by 
\begin{itemize}
\item[(A2')] $\|a\|_{L^{r}_{\mathbb{B}}(\rnp)} \le 
	(\mu_{\lam}(B(x_0,r_0)))^{1\slash r -1}$ for some fixed $1\le r < \infty$.
\end{itemize}
In the one-dimensional case and for $\mathbb{B}=\mathbb{C}$ the atomic Hardy space in the Bessel
setting was investigated in \cite{BDT,D,YY}.

The space $\textrm{BMO}_{\mathbb{B}}(d\mu_{\lam})$ is defined to
consist of all locally $d\mu_{\lam}$-integrable $\mathbb{B}$-valued functions $f$ on $\rnp$ such that 
$$
\|f\|_{\textrm{BMO}_{\mathbb{B}}(d\mu_{\lam})} = 
\sup \frac{1}{\mu_{\lam}(B)}\int_{B}\|f(x)-f_B\|_{\mathbb{B}}\,d\mu_{\lam}(x) < \infty,
$$
where the supremum is taken over all balls $B$ in $(\rnp,d\mu_{\lam},|\cdot|)$, and $f_B$ is the
mean value of $f$ over $B$, $f_B = \frac{1}{\mu_{\lam}(B)}\int_B f(y)\, d\mu_{\lam}(y)$.
When $\mathbb{B}$ satisfies certain mild conditions, 
the dual of $H_{\mathbb{B},\textrm{at}}^{1,\lam}$ is $\textrm{BMO}_{\mathbb{B}^*}(d\mu_{\lam})$,
with $\mathbb{B}^*$ being the dual of $\mathbb{B}$.

It is well known that a large part of the classical theory of Calder\'on-Zygmund operators remains valid,
with appropriate adjustments, when the underlying space is of homogeneous type and the associated kernels
are vector-valued, see for instance \cite{RRT,RT}. In particular, if $T$ is a Calder\'on-Zygmund
operator in the sense of $(\rnp,d\mu_{\lam},|\cdot|)$ associated with a Banach space $\mathbb{B}$,
then (see \cite[Theorem 1.1]{BFS} and references given there)
\begin{itemize}
\item[(M1)] $T$ extends to a bounded operator from $L^p(wd\mu_{\lam})$ to $L^p_{\mathbb{B}}(wd\mu_{\lam})$,
	for every $1< p < \infty$ and every $w\in A_p^{\lam}$,
\item[(M2)] $T$ extends to a bounded operator from $L^1(wd\mu_{\lam})$ to
 	weak $L^1_{\mathbb{B}}(wd\mu_{\lam})$,
	for every $w\in A_1^{\lam}$,
\item[(M3)] $T$ extends to a bounded operator from $H_{\mathbb{C},\textrm{at}}^{1,\lam}$ to
	$L^1_{\mathbb{B}}(d\mu_{\lam})$,
\item[(M4)] $T$ extends to a bounded operator from $L^{\infty}_c$ to
	$\textrm{BMO}_{\mathbb{B}}(d\mu_{\lam})$.
\end{itemize}
In (M1)-(M4) it is implicitly assumed that $T$ is given initially on dense subspaces being intersections
of the relevant spaces with $L^r(d\mu_{\lam})$.

The main result of the paper reads as follows.
\begin{Th} \label{thm:main}
Let $\lambda \in [0,\infty)^n$, $m \in \mathbb{N}^n$, $k \ge 0$, $k + |m|>0$, $2\le r<\infty$, 
$\mathcal{M} = \mathcal{M}_W$ or $\mathcal{M} = \mathcal{M}_P$. Then each of the operators
$$
W_*^{\lam}, \; P_*^{\lam}, \; g^{\lam,W}_{m,k,r}, \; g^{\lam,P}_{m,k,r}, \; 	
T^{\lam}_{\mathcal{M}}, \; R^{\lam}_m,
$$
can be interpreted as a Calder\'on-Zygmund operator in the sense of the space of homogeneous type
$(\rnp,d\mu_{\lam},|\cdot|)$ associated with a Banach space $\mathbb{B}$, where $\mathbb{B}$
is $C_0$, $C_0$, $L^r(t^{(k+|m|\slash 2)r-1}dt)$, $L^r(t^{(k+|m|)r-1}dt)$, 
$\mathbb{C}$, $\mathbb{C}$, respectively.
\end{Th}

Consequently, each of the operators listed above satisfies (M1)-(M4), with appropriately chosen
$\mathbb{B}$ in each case. 
When $\mathbb{B}\neq \mathbb{C}$, these mapping properties can be translated to the following.

\begin{Cor} \label{cor:main}
Let $\lambda \in [0,\infty)^n$, $m \in \mathbb{N}^n$, $k \ge 0$, $k + |m|>0$ and $2\le r<\infty$.
Then the operators 
$W_*^{\lam}, P_*^{\lam}, g^{\lam,W}_{m,k,r}, g^{\lam,P}_{m,k,r}$,
viewed as
scalar-valued operators, satisfy (M1)-(M4) with $\mathbb{B}=\mathbb{C}$. Moreover, the 
resulting extensions are given by the formulas defining initially these operators
in $L^r(d\mu_{\lam})$ (with $r=2$ in case of the maximal operators), 
where the relevant integrals converge. 
\end{Cor}

\begin{proof}
The first statement is straightforward. The second one is justified by standard arguments, 
with the aid of Remark \ref{rem:conv} below;
see for instance the proofs of \cite[Theorem 2.1]{NS2} and \cite[Corollary 2.5]{Szarek}.
\end{proof}

The proof of Theorem \ref{thm:main} will be given in Section \ref{sec:proof}. 
Treatments of each of the operators are naturally divided into the following three steps.
\begin{itemize}
\item[\textbf{Step 1.}] The operator is bounded from $L^r(d\mu_{\lam})$ to $L^r_{\mathbb{B}}(d\mu_{\lam})$
	for some $1<r<\infty$.
\item[\textbf{Step 2.}] The operator is associated with an integral kernel in the sense of (b) above.
\item[\textbf{Step 3.}] The kernel satisfies the standard estimates \eqref{gr}, \eqref{sm1} and \eqref{sm2}.
\end{itemize}

In the next section we gather various facts and lemmas that will be needed in the proof of 
Theorem \ref{thm:main}.

\section{Preparatory facts and results} \label{sec:prep}

The modified Bessel function $I_{\nu}$ has the following Poisson-type integral representation
obtained by Schl\"afli, see \cite[Chapter VI, Section 6$\cdot$15]{Wat}. For $\nu > -1\slash 2$,
\begin{equation} \label{Schlafli}
I_{\nu}(z) = z^{\nu} \int_{-1}^1 e^{-zs}\, d\Omega_{\nu+1\slash 2}(s), \qquad z >0,
\end{equation}
where $\Omega_{\eta}$ is the measure on $[-1,1]$ given by the density
$$
d\Omega_{\eta}(s) = \frac{(1-s^2)^{\eta-1}\, ds}{\sqrt{\pi}2^{\eta-1\slash 2}\Gamma(\eta)}, \qquad \eta>0.
$$
In the limit case $\eta=0$ we put $\Omega_0 = (\delta_{-1}+\delta_{1})\slash \sqrt{2\pi}$, where
$\delta_{-1}$ and $\delta_1$ are the point masses at $-1$ and $1$, respectively, so that \eqref{Schlafli}
holds also for $\nu=-1\slash 2$. Then, for any $\lambda \in [0,\infty)^n$, the Bessel heat kernel can be
written as 
\begin{equation} \label{heatform}
W_t^{\lam}(x,y) = \frac{1}{(2t)^{n\slash 2 +|\lam|}} \int \exp\Big( -\frac{1}{4t}q(x,y,s)\Big)
	\, d\Omega_{\lam}(s), \qquad x,y \in \rnp, \quad t>0,
\end{equation}
where $|\lam|=\lam_1+\ldots+\lam_n$, the function $q$ is given by
$$
q(x,y,s) = |x|^2+|y|^2 + 2\sum_{j=1}^n x_j y_j s_j, \qquad x,y \in \rnp, \quad s \in [-1,1]^n,
$$
and $\Omega_{\lam}$ denotes the product measure $\bigotimes_{i=1}^n \Omega_{\lam_i}$
on the cube $[-1,1]^n$.

The following result is a crucial point in our method of estimating kernels.
It relates expressions involving certain integrals with respect to $d\Omega_{\lam}$ with the
standard estimates for the space $(\rnp, d\mu_{\lam}, |\cdot|)$. 
\begin{Lem}[{\cite[Proposition 5.9]{NS2}}] \label{lem:bridge}
Assume that $\lam \in [0,\infty)^n$. Then
\begin{align*}
\int \big( q(x,y,s) \big)^{-n\slash 2 - |\lam|} \, d\Omega_{\lam}(s) & \lesssim
	\frac{1}{\mu_{\lam}(B(x,|x-y|))}, \qquad x \neq y, \\
\int \big( q(x,y,s) \big)^{-n\slash 2 - |\lam|-1\slash 2} \, d\Omega_{\lam}(s) & \lesssim
	\frac{1}{|x-y|\mu_{\lam}(B(x,|x-y|))}, \qquad x \neq y.
\end{align*}
\end{Lem}
At this point it is perhaps interesting to observe that, see \cite[Proposition 3.2]{NS2}, 
$$
\mu_{\lam}(B(x,R)) \simeq R^n \prod_{j=1}^n (x_j + R)^{2\lambda_j}, \qquad x \in \rnp, \quad R>0.
$$

The result below will come into play when proving the smoothness estimates \eqref{sm1} and \eqref{sm2}
in cases when $\mathbb{B}\neq \mathbb{C}$. It will enable us to reduce the difference conditions
to certain gradient estimates, which are easier to verify.
\begin{Lem}[{\cite[Lemma 4.5]{Szarek}}] \label{lem:theta}
If $x,x',y\in \rnp$ are such that $|x-y|>2|x-x'|$ and $\theta = \alpha x + (1-\alpha) x'$ for some
$\alpha \in [0,1]$, then
$$
\frac{1}{4} q(x,y,s) \le q(\theta,y,s) \le 4 q(x,y,s), \qquad s \in [-1,1]^n.
$$
The same holds after exchanging the roles of $x$ and $y$.
\end{Lem}

The following technical result will be used repeatedly while showing the relevant kernel estimates.
To prove it, and also for future use, we introduce some additional notation.
Given two multi-indices $m,M \in \mathbb{N}^n$, the relation $M \le m$ means that $M_i\le m_i$ for
all $i=1,\ldots,n$. We write $(\partial_x q)^{m}$ to denote the quantity
$(\partial_{x_1}q)^{m_1} \cdot \ldots \cdot (\partial_{x_n}q)^{m_n}$.
\begin{Lem} \label{lem:estexp}
Let $A>0$, $m,r \in \mathbb{N}^n$ and $k \in \mathbb{N}$ be fixed. Then
$$
\bigg| \partial_t^k\partial_x^m \partial_y^r \bigg( t^{-A} \exp\Big( -\frac{1}{4t} q(x,y,s)\Big)\bigg)
	\bigg| \lesssim t^{-A-k-(|m|+|r|)\slash 2} \exp\Big( -\frac{1}{8t} q(x,y,s)\Big),
$$
uniformly in $x,y \in \rnp$, $t>0$ and $s\in [-1,1]^n$.
\end{Lem}

\begin{proof}
Taking into account the inequality
$$
|a+bs| \le (a^2+b^2+2abs)^{1\slash 2}, \qquad a,b \ge 0, \quad s \in [-1,1],
$$
we see that it is sufficient to show that
\begin{align*}
& \partial_t^k\partial_x^m \partial_y^r \bigg( t^{-A} \exp\Big( -\frac{1}{4t} q(x,y,s)\Big)\bigg)
= \sum C_{j,M,R}(s)\, t^{-A-k-j-(|m|+|r|+|M|+|R|)\slash 2} \\
&	\qquad \times \big(\partial_x q(x,y,s)\big)^M
	\big( \partial_y q(x,y,s)\big)^{R} \big(q(x,y,s)\big)^{j} \exp\Big( -\frac{1}{4t} q(x,y,s)\Big),
\end{align*}
where the (finite) summation runs over $0 \le j \le k$, $M \le m$, $R \le r$, and $C_{j,M,R}$
are polynomials.

Proceeding inductively, we arrive at the formula
\begin{equation} \label{p1}
\partial_t^{k} \bigg( t^{-A} \exp\Big( -\frac{1}{4t} q(x,y,s)\Big) \bigg)
 = \sum_{0 \le j \le k} \alpha_{j,k} \, t^{-A-k-j} \big( q(x,y,s) \big)^{j}
 \exp\Big( -\frac{1}{4t} q(x,y,s)\Big),
\end{equation}
where $\alpha_{j,k} \in \mathbb{R}$. On the other hand, given $i,j=1,\ldots,n$,
another inductive reasoning leads to
\begin{align} \label{p2}
\partial_{x_i}^{m_i} \partial_{y_j}^{r_j} \exp\Big( -\frac{1}{4t} q(x,y,s)\Big)  = &
\sum_{\substack{0 \le M_i \le m_i\\ 0 \le R_j \le r_j}} 
\beta_{M_i,R_j}(s_i,s_j) \, t^{-(m_i+r_j+M_i+R_j)\slash 2} \\ & \quad \times
\big( \partial_{x_i}q(x,y,s)\big)^{M_i} \big(\partial_{y_j} q(x,y,s)\big)^{R_j}
\exp\Big( -\frac{1}{4t} q(x,y,s)\Big), \nonumber
\end{align}
where
$\beta_{M_i,R_j}(s_i,s_j)$ are polynomials of two variables (some of them being null polynomials).

The conclusion follows by combining \eqref{p1} with \eqref{p2}.
\end{proof}

The next result provides a useful decomposition of $\partial_x^{m}\varphi_x^{\lam}$.
This will be necessary to prove $L^2$-boundedness of operators involving
higher order `horizontal' derivatives, that is the $g$-functions and the Riesz transforms.
\begin{Lem} \label{decomp_phi}
Let $\lam \in [0,\infty)^n$. Given $m \in \mathbb{N}^n$, there exist numbers $c_j \in \mathbb{R}$, 
$j \le m$, such that
\begin{equation} \label{decomp}
\partial_x^{m} \varphi_x^{\lam}(z) = z^m \sum_{j \le m} c_j (xz)^j \varphi_x^{\lambda+j}(z),
	\qquad x,z \in \rnp,
\end{equation}
with the notation $z^m = z_1^{m_1}\cdot \ldots \cdot z_n^{m_n}$ and
$(xz)^{j} = (x_1 z_1)^{j_1}\cdot \ldots \cdot (x_n z_n)^{j_n}$.
\end{Lem}

\begin{proof}
By the tensor product structure of $\varphi_x^{\lam}$ it is enough to prove the result
in the one-dimensional case. Then we proceed by induction on $m$.
Using the identity  (cf. \cite[Chapter III, Section 3$\cdot$2]{Wat})
$$
\frac{d}{dz} \big( z^{-\nu} J_{\nu}(z)\big) = - z^{-\nu}J_{\nu+1}(z), 
$$
we get $\partial_x \varphi_x^{\lam}(z) = - z (xz)\varphi_{x}^{\lam+1}(z)$.
Suppose that \eqref{decomp} holds for certain $m \in \mathbb{N}$. Then, with the aid of the recurrence
relation (cf. \cite[Chapter III, Section 3$\cdot$2]{Wat})
$$
J_{\nu}(z) = \frac{z}{2\nu} \big( J_{\nu-1}(z) + J_{\nu+1}(z)\big), \qquad \nu > 0,
$$
it follows that
\begin{align*}
& \partial^{m+1}_x\varphi_x^{\lam}(z) = z^m \bigg( \sum_{j=1}^m j c_j z(xz)^{j-1} \varphi_{x}^{\lam+j}(z)
	- \sum_{j=0}^m c_j z (xz)^{j+1} \varphi_{x}^{\lam+j+1}(z) \bigg) \\
& = z^{m+1} \bigg(\sum_{j=1}^m \frac{j c_j}{2\lam+2j-1}\Big( (xz)^{j-1}\varphi_x^{\lam+j-1}(z)
	+ (xz)^{j+1} \varphi_x^{\lam+j+1}(z) \Big) - \sum_{j=0}^m c_j (xz)^{j+1} \varphi_x^{\lam+j+1}(z)\bigg).
\end{align*}
Thus we conclude that \eqref{decomp} is satisfied with $m$ replaced by $m+1$.
\end{proof}

To prove the remaining two lemmas we will need the estimate
\begin{equation} \label{phiest}
\big|\varphi_x^{\lam}(z)\big| \lesssim 1, \qquad x,z \in \rnp.
\end{equation}
This is easily obtained from the following basic asymptotics for the Bessel function $J_{\nu}$
(cf. \cite[Chapter III, Section 3$\cdot$1 (8), Chapter VII, Section 7$\cdot$21]{Wat}): 
for $\nu > -1$ we have
$$
J_{\nu}(z) \simeq z^{\nu}, \quad z \to 0^+, \qquad
J_{\nu}(z) = \mathcal{O}\Big( \frac{1}{\sqrt{z}}\Big), \quad z \to \infty.
$$

\begin{Lem} \label{lem:heat}
Let $\lam \in [0,\infty)^n$ and assume that $f \in L^p(wd\mu_{\lam})$, 
$w \in A_p^{\lam}$, $1\le p < \infty$. Then the heat integral of $f$,
$$
W_t^{\lam}f(x) = \int_{\rnp} W_t^{\lam}(x,y)f(y)\, d\mu_{\lam}(y), \qquad x \in \rnp, \quad t>0,
$$
is a well-defined $C^{\infty}$ function of $(x,t) \in \rnp \times \rp$. Moreover, given 
$m \in \mathbb{N}^n$ and $k \in \mathbb{N}$, we have 
$$
\partial_x^m \partial_t^k W_t^{\lam}f(x) = \int_{\rnp} \partial_x^m \partial_t^k W_t^{\lam}(x,y)
	f(y)\, d\mu_{\lam}(y), \qquad x \in \rnp, \quad t>0.
$$
Furthermore, if $f \in L^2(d\mu_{\lam})$, then we also have
$$
\partial_x^m \partial_t^k W_t^{\lam}f(x) = \int_{\rnp} \partial_x^m \partial_t^k
	\Big( e^{-t|z|^2} \varphi_x^{\lam}(z)\Big) h_{\lam}f(z) \, d\mu_{\lam}(z), 
	\qquad x \in \rnp, \quad t>0.
$$
\end{Lem}

\begin{proof}
Let $E \subset \overline{E}\subset \rnp$ and $F \subset \overline{F} \subset \rp$ be bounded subsets
of $\rnp$ and $\rp$, respectively. Observe first that by \eqref{heatform} and the fact that
$q(x,y,s) \ge |x-y|^2$, $s\in [-1,1]^n$, we have
$$
W_t^{\lam}(x,y) \lesssim \frac{1}{t^{n\slash 2 + |\lam|}} \exp\Big(- \frac{1}{4t}|x-y|^2\Big),
	\qquad x,y \in \rnp, \quad t>0.
$$
This implies the estimate
$$
W_t^{\lam}(x,y) \lesssim e^{-c|y|^2}, \qquad y \in \rnp, \quad x \in E, \quad t \in F,
$$
where $c>0$ is a constant depending, in particular, on $E$ and $F$. Moreover, combining \eqref{heatform}
with Lemma \ref{lem:estexp} we see that, given $m\in \mathbb{N}^n$ and $k \in \mathbb{N}$,
\begin{equation} \label{differ}
\big| \partial_x^m \partial_t^k W_t^{\lam}(x,y) \big| \lesssim 
	e^{-c|y|^2}, \qquad y \in \rnp, \quad x \in E, \quad t \in F.
\end{equation}

Next we observe that the function $y \mapsto e^{-c|y|^2}$ belongs to all $L^p(wd\mu_{\lam})$,
$w\in A_p^{\lam}$, $1\le p < \infty$. This can be easily verified by splitting $\rnp$ into
dyadic `rings',
$$
\rnp = \{x \in \rnp : |x| < 1\} \cup \bigcup_{j=0}^{\infty} \big\{x\in \rnp: 2^{j} \le |x| < 2^{j+1}\big\}
$$
and using the estimate
$$
w(B(0,R)) = \int_{B(0,R)} w(x) \, d\mu_{\lam}(x) \lesssim R^{(n+2|\lam|)p}, \qquad R \ge 1,
$$
see for instance \cite[Section 4]{N}. Similarly one can show that the function 
$y \mapsto e^{-c|y|^2}\slash w(y)$ is essentially bounded if $w \in A_1^{\lam}$, and to do that
one uses the dyadic decomposition above and the estimate
$$
\|\chi_{B(0,R)} w^{-1}\|_{\infty} \lesssim R^{n+2|\lam|}, \qquad R \ge 1,
$$
following from the $A_1^{\lam}$ condition, see \cite[Section 4]{N}.

Now let $f \in L^p(wd\mu_{\lam})$, $w \in A_p^{\lam}$, $1\le p < \infty$.
By the above observations and H\"older's inequality we can write
\begin{align*}
\int_{\rnp} W_t^{\lam}(x,y) |f(y)|\, d\mu_{\lam}(y) & 
	\lesssim \int_{\rnp} e^{-c|y|^2} |f(y)|\, d\mu_{\lam}(y) \\
& \le \|f\|_{L^p(wd\mu_{\lam})} \big\| e^{-c|\cdot|^2}\big\|_{L^{p'}(w^{-p'\slash p}d\mu_{\lam})}, \qquad
	x \in E, \quad t \in F,
\end{align*}
and the last expression is finite since $w^{-p'\slash p} \in A_{p'}^{\lam}$ (for $p=1$ the $L^{p'}$
norm here must be replaced by $\|e^{-c|\cdot|^2}w^{-1}\|_{\infty}$).
Thus $W_t^{\lam}$ is well defined for functions from the weighted $L^p$ spaces.
The fact that $W_t^{\lam}f(x)$ is a smooth function of $(x,t)\in \rnp\times\rp$ and the possibility
of passing with the derivatives $\partial_x^m \partial_t^k$ under the integral sign are proved
inductively, by considering one partial derivative at a time and then using \eqref{differ} and the
dominated convergence theorem.
Here, and also elsewhere, the possibility of applying this theorem for differentiating 
under integral signs is justified with the aid of the Mean Value Theorem. 

Finally, we verify the last statement of the lemma. For $f \in L^2(d\mu_{\lam})$ we have
$$
W_t^{\lam}f(x) = h_{\lam}\big(e^{-t|z|^2}h_{\lam}f\big)(x) =
	\int_{\rnp} e^{-t|z|^2} \varphi_z^{\lam}(x) h_{\lam}f(z) \, d\mu_{\lam}(z), \qquad x \in \rnp.
$$
Further, by Lemma \ref{decomp_phi} and \eqref{phiest}
$$
\big| \partial_x^m \partial_t^k \big( e^{-t|z|^2}\varphi_x^{\lam}(z)\big)\big| \lesssim
	e^{-c|z|^2}, \qquad z \in \rnp, \quad x \in E, \quad t \in F,
$$
for some constant $c>0$, and the function $z \mapsto e^{-c|z|^2}h_{\lam}f(z)$ is integrable
against $d\mu_{\lam}$, as can be easily seen by means of H\"older's inequality and the
$L^2$-boundedness of $h_{\lam}$. In this position we can proceed inductively as before.
\end{proof}

\begin{Rem} \label{rem:conv}
From the proof of Lemma \ref{lem:heat} we can conclude immediately the following convergence
result. Assume that $\lam \in [0,\infty)^n$ and $f_N \to f$ in $L^p(wd\mu_{\lam})$ for some
$w\in A_p^{\lam}$ and $1\le p < \infty$. Let $m \in \mathbb{N}^n$ and $k \in \mathbb{N}$ be given.
Then, with $t>0$ fixed, 
$$
\partial_x^m \partial_t^k W_t^{\lam}f_N \to \partial_x^m \partial_t^k W_t^{\lam}f, \qquad 
	\textrm{as}\;\; N\to \infty,
$$
pointwise, and even uniformly on bounded subsets $E \subset \overline{E}\subset \rnp$.
\end{Rem}

\begin{Lem} \label{lem:hsmooth}
Let $\lam \in [0,\infty)^d$ and assume that $f \in C_c^{\infty}(\rnp)$. 
Then $h_{\lam}f \in C^{\infty}(\rnp)$ and, given $m \in \mathbb{N}^n$,
$$
\partial_z^m h_{\lam}f(z) = \int_{\rnp} \partial_z^m \varphi_z^{\lam}(x)f(x) \, d\mu_{\lam}(x).
$$
\end{Lem}

\begin{proof}
The arguments are based on Lemma \ref{decomp_phi}, the estimate \eqref{phiest} and the dominated
convergence theorem. We leave elementary details to the reader.
\end{proof}

\section{Proof of Theorem~\ref{thm:main}} \label{sec:proof}

In this section we prove our main result, Theorem \ref{thm:main}.
In the following subsections we treat separately the cases of the maximal operators, the $g$-functions,
the Laplace transform type multipliers and finally the Riesz transforms.

In the sequel we will often omit the arguments and write shortly $\q$ instead of $q(x,y,s)$.
We will also make a frequent use, without further mentioning, of the fact that for $A\ge 0$ and
$B > 0$, $\sup_{t>0} t^A e^{-Bt}= C_{A,B}< \infty$.
We shall sometimes tacitly assume that passing with the differentiation in $x$, or $y$ or $t$,
under the integral against $d\Omega_{\lam}$ or against $dt$ is legitimate;
similarly for changing orders of integrals. This is indeed the case in all the relevant cases,
which may be verified in a straightforward manner by means of the estimates obtained in the proof of
Theorem \ref{thm:main} and the dominated convergence theorem.
On the other hand, in more subtle cases we will always comment in detail operations of this kind.

\subsection{Maximal operators $W_*^{\lam}$ and $P_*^{\lam}$} \quad\\
First, recall that $\{W_t^{\lam}\}$ is a symmetric diffusion semigroup in the sense of Stein's
monograph \cite{Stlitt}. 
Then the maximal theorem \cite[p.\,73]{Stlitt} applies showing the boundedness of $W^{\lam}_*$
on $L^p(d\mu_{\lam})$, $1< p \le \infty$. The same is true for $P^{\lam}_*$, by subordination.
Theorem \ref{thm:main} complements these results by admitting weights and providing some further
mapping properties of the maximal operators, see Corollary \ref{cor:main}.

We treat in detail only the heat semigroup maximal operator and then only indicate how to make the arguments
go through in case of $P^{\lam}_*$.

\noindent \textbf{Step 1.} 
We want to view $W^{\lam}_*$ as a vector-valued operator $\mathcal{W}^{\lam}$,
assigning to any $f \in L^2(d\mu_{\lam})$ the function
$$
\rnp \ni x \mapsto \mathcal{W}^{\lam}f(x) = \{W_t^{\lam}f(x)\}_{t>0},
$$
and bounded from $L^2(d\mu_{\lam})$ to $L^2_{C_0}(d\mu_{\lam})$.
We first ensure that $\mathcal{W}^{\lam}$ indeed takes its values in $L^2_{C_0}(d\mu_{\lam})$.
Here the arguments are analogous to those from the proof of \cite[Theorem 2.1]{NS2}. 
To make them work one needs two ingredients. The first one is the fact that, 
given $f \in L^2(d\mu_{\lam})$ and $x\in \rnp$, 
the function $W_t^{\lam}f(x)$ is continuous in $t\in (0,\infty)$, which we already know 
(see Lemma \ref{lem:heat}).
To get the remaining ingredient, it is enough to check that for $f \in L^2(d\mu_{\lam})$
\begin{equation} \label{limits}
\lim_{t \to 0^+} W_t^{\lam}f(x) = f(x), \qquad \lim_{t\to \infty} W_t^{\lam}f(x)=0, \qquad
	\textrm{a.e.}\;\; x\in \rnp.
\end{equation} 
This, however, is rather straightforward. We have 
$$
W_t^{\lam}f = h_{\lam}(e^{-t|z|^2}h_{\lam}f), \qquad f \in L^2(d\mu_{\lam}), \quad t>0.
$$
Then \eqref{limits} for $f \in C^{\lam}$ follows by the dominated convergence theorem.
Since $C^{\lam}$ is dense in $L^2(d\mu_{\lam})$ (see Section \ref{sec:riesz} below)
and the maximal operator $W_*^{\lam}$ is bounded
on $L^2(d\mu_{\lam})$, standard arguments show that \eqref{limits} holds for $f \in L^2(d\mu_{\lam})$.

Thus $\mathcal{W}^{\lam}$ is a linear mapping from $L^2(d\mu_{\lam})$ to $L^2_{C_0}(d\mu_{\lam})$, and
as such is bounded, by the corresponding property of the scalar-valued operator $W_*^{\lam}$.

\noindent \textbf{Step 2.}
The fact that $\mathcal{W}^{\lam}$ is associated with the vector-valued kernel 
$\{W_t^{\lam}(x,y)\}_{t>0}$ is justified exactly in the same way as the corresponding fact in 
the proof of \cite[Theorem 2.1]{NS2}.

\noindent \textbf{Step 3.}
We prove the standard estimates for the kernel $\{W_t^{\lam}(x,y)\}_{t>0}$.
By \eqref{heatform} we have
\begin{align*}
\|W_t^{\lam}(x,y)\|_{L^{\infty}(dt)} & = \sup_{t>0} \frac{1}{(2t)^{n\slash 2+|\lam|}} \int
	\exp\Big( -\frac{\q}{4t}\Big) \, d\Omega_{\lam}(s) \\
& \le \int \frac{1}{\q^{n\slash 2+|\lam|}} \sup_{t>0} \Big(\frac{\q}{t}\Big)^{n\slash 2+|\lam|}
	\exp\Big( -\frac{\q}{4t}\Big) \, d\Omega_{\lam}(s) \lesssim 
	\int \frac{1}{\q^{n\slash 2+|\lam|}} \, d\Omega_{\lam}(s)
\end{align*}
and the growth estimate \eqref{gr} with $\mathbb{B}=C_0$
follows from Lemma \ref{lem:bridge}.

To show the smoothness estimates \eqref{sm1} and \eqref{sm2} it is enough, by symmetry reasons,
to consider \eqref{sm1}. Then using the Mean Value Theorem we get
$$
\big| W_t^{\lam}(x,y)-W_t^{\lam}(x',y)\big| \le 
	|x-x'| \Big| \nabla_{\! x} W_t^{\lam}(x,y)\big|_{x=\theta}\Big|,
$$
where $\theta$ is a convex combination of $x$ and $x'$ that depends also on $t$.
Thus it suffices to show the estimates
$$
\big\| \partial_{x_i} W_t^{\lam}(x,y)\big|_{x=\theta}\big\|_{L^{\infty}(dt)} \lesssim 
	\frac{1}{|x-y|\mu_{\lam}(B(x,|x-y|))}, \qquad i=1,\ldots,n,
$$
for all $x,x',y$ satisfying $|x-y|>2|x-x'|$.
Using \eqref{heatform} and applying Lemma \ref{lem:estexp} we get
\begin{align*}
|\partial_{x_i} W_t^{\lam}(x,y)| & \le \frac{1}{(2t)^{n\slash 2 + |\lam|}} \int \Big|\partial_{x_i}
	\exp\Big( -\frac{\q}{4t}\Big)\Big| \, d\Omega_{\lam}(s) \\
& \lesssim \frac{1}{t^{n\slash 2+|\lam|+1\slash 2}} \int \exp\Big( -\frac{\q}{8t}\Big) \, d\Omega_{\lam}(s)
\lesssim \int \frac{1}{\q^{n\slash 2+|\lam|+1\slash 2}} \, d\Omega_{\lam}(s).
\end{align*}
Now the desired bound follows by Lemma \ref{lem:theta} and Lemma \ref{lem:bridge}. This finishes Step 3.

Treatment of $P_*^{\lam}$ goes along the same lines. One has to combine the arguments given above
with the subordination principle. In Step 1 the relevant identity is 
$P_t^{\lam}f = h_{\lam}(e^{-t|z|}h_{\lam}f)$, $f \in L^2(d\mu_{\lam})$. 
We leave further details to interested readers.

\subsection{Square functions $g^{\lam,W}_{m,k,r}$ and $g^{\lam,P}_{m,k,r}$} \quad\\
We analyse in detail only the square function $g^{\lam,W}_{m,k,r}$ based on the heat semigroup.
The Poisson semigroup based $g$-function is treated in a similar way, by means of the subordination
principle (see for instance \cite[Section 4.3]{Szarek}), hence the details are omitted.

\noindent \textbf{Step 1.} We interpret $g^{\lam,W}_{m,k,r}$ as the vector-valued operator
$\mathcal{G}^{\lam,W}_{m,k}$, assigning to any $f \in L^r(d\mu_{\lam})$ the function
$$
\rnp \ni x \mapsto \mathcal{G}^{\lam,W}_{m,k}f(x) = \big\{ \partial_x^m \partial_t^k W_t^{\lam}f(x)
	\big\}_{t>0}.
$$
We will show that $\mathcal{G}^{\lam,W}_{m,k}$ is bounded from $L^r(d\mu_{\lam})$ to
$L^r_{\mathbb{B}}(d\mu_{\lam})$, $2\le r < \infty$, where $\mathbb{B}= L^r(t^{(k+|m|\slash 2)r-1}dt)$.
Actually, this task amounts to showing that the $g$-function $g^{\lam,W}_{m,k,r}$ is bounded on
$L^r(d\mu_{\lam})$, $2\le r < \infty$. 

We first deal with the case $r=2$. Let $f\in L^2(d\mu_{\lam})$. By Lemma \ref{lem:heat} and 
Lemma \ref{decomp_phi} we see that
$$
\partial_x^m \partial_t^k W_t^{\lam}f(x) = \sum_{j \le m} c_j \int_{\rnp} (xz)^j \varphi_x^{\lam+j}(z)
	z^m (-1)^k |z|^{2k} e^{-t|z|^2} h_{\lam}f(z)\, d\mu_{\lam}(z).
$$
Then
\begin{align*}
&\big\| g^{\lam,W}_{m,k,2}(f)\big\|^2_{L^2(d\mu_{\lam})}  \\
& = \int_{\rnp} \int_0^{\infty} \bigg| \sum_{j \le m}
	c_j \int_{\rnp} (xz)^j \varphi_x^{\lam+j}(z) z^m (-1)^k |z|^{2k} e^{-t|z|^2} h_{\lam}f(z)\, 
	d\mu_{\lam}(z)\bigg|^2 \,t^{2k+|m|-1}dt\, d\mu_{\lam}(x) \\
& \lesssim 	\sum_{j \le m} \int_0^{\infty} \int_{\rnp} \bigg| 
	 \int_{\rnp} (xz)^j \varphi_x^{\lam+j}(z) z^m (-1)^k |z|^{2k} e^{-t|z|^2} h_{\lam}f(z)\, 
	d\mu_{\lam}(z)\bigg|^2 \,d\mu_{\lam}(x) \,t^{2k+|m|-1}dt  \\
& = \sum_{j \le m} \int_0^{\infty} \int_{\rnp} \big| h_{\lam+j}\big( z^{m-j} |z|^{2k} e^{-t|z|^2}
	 h_{\lam}f\big)(x)\big|^2 \,d\mu_{\lam+j}(x)\, t^{2k+|m|-1}dt.
\end{align*}
Since $h_{\lam+j}$ is an isometry on $L^2(d\mu_{\lam+j})$, this implies
$$
\big\| g^{\lam,W}_{m,k,2}(f)\big\|^2_{L^2(d\mu_{\lam})} \lesssim \int_{\rnp} x^{2m} |x|^{4k}
	\int_0^{\infty} e^{-2t|x|^2} t^{2k+|m|-1}dt \, |h_{\lam}f(x)|^2\, d\mu_{\lam}(x).
$$
The integral in $t$ here is equal $\Gamma(2k+|m|)\slash (2|x|^2)^{2k+|m|}$ and we conclude that
$$
\big\| g^{\lam,W}_{m,k,2}(f)\big\|^2_{L^2(d\mu_{\lam})} 
\lesssim \int_{\rnp} \frac{x^{2m}|x|^{4k}}{|x|^{4k+2|m|}}
	|h_{\lam}f(x)|^2 \, d\mu_{\lam}(x) \le \|h_{\lam}f\|^2_{L^2(d\mu_{\lam})} = \|f\|^2_{L^2(d\mu_{\lam})}.
$$
Notice that when there is no horizontal component, i.e. $m=(0,\ldots,0)$, then 
$\| g^{\lam,W}_{m,k,2}(f)\|_{L^2(d\mu_{\lam})}$ is equal to $\|f\|_{L^2(d\mu_{\lam})}$, up to
a factor independent of $f$.

Considering the boundedness for $2<r<\infty$, it is enough to show that it follows from the 
$L^r$-boundedness of $g^{\lam,W}_{m,k,2}$ and the maximal operator $W^{\lam}_{*}$.
To begin with, observe that by \eqref{heatform} and Lemma \ref{lem:estexp}
$$
\big| \partial_x^m \partial_t^k W_t^{\lam}(x,y)\big| \lesssim \frac{1}{t^{n\slash 2+|\lam|+k+|m|\slash 2}}
	\int \exp\Big( -\frac{\q}{8t}\Big) \, d\Omega_{\lam}(s),
$$
so $t^{k + |m|\slash 2} |\partial_x^m \partial_t^k W_t^{\lam}(x,y)| \lesssim W_{2t}^{\lam}(x,y)$.
Consequently, for suitable $f$,
$$
t^{k+|m|\slash 2} \big| \partial_x^m \partial_t^k W_t^{\lam}f(x)\big| \lesssim W_{2t}^{\lam}|f|(x),
	\qquad x \in \rnp, \quad t>0.
$$
Using this observation and H\"older's inequality we obtain, for $f \in L^r(d\mu_{\lam})$,
\begin{align*}
& \big\| g^{\lam,W}_{m,k,r}(f)\big\|^r_{L^r(d\mu_{\lam})} \\
& = \int_{\rnp} \int_0^{\infty} \big|
	t^{k+|m|\slash 2} \partial_x^m \partial_t^k W_t^{\lam}f(x)\big|^r \, \frac{dt}t \, d\mu_{\lam}(x) \\
& \le \int_{\rnp} \int_0^{\infty} \big|
	t^{k+|m|\slash 2} \partial_x^m \partial_t^k W_t^{\lam}f(x)\big|^2 \, \frac{dt}t \Big( \sup_{t>0}
	t^{k+|m|\slash 2} \big|\partial_x^m \partial_t^k W_t^{\lam}f(x)\big|\Big)^{r-2}\, d\mu_{\lam}(x) \\
& \le \bigg( \int_{\rnp} \bigg( \int_0^{\infty} \big|
	t^{k+|m|\slash 2} \partial_x^m \partial_t^k W_t^{\lam}f(x)\big|^2 \, \frac{dt}t \bigg)^{r\slash 2}\,
	d\mu_{\lam}(x)\bigg)^{2\slash r} \\
& \qquad \times	\bigg( \int_{\rnp} \Big( \sup_{t>0}
	t^{k+|m|\slash 2} \big|\partial_x^m \partial_t^k W_t^{\lam}f(x)\big|\Big)^{r} \, d\mu_{\lam}(x)
	\bigg)^{1-2\slash r} \\
& \lesssim \big\| g^{\lam,W}_{m,k,2}(f)\big\|^2_{L^r(d\mu_{\lam})} \big\|W_*^{\lam}|f|\big\|^{r-2}_{L^r(d\mu_{\lam})}.
\end{align*}
Thus if $g^{\lam,W}_{m,k,2}$ and $W_{*}^{\lam}$ are bounded on $L^r(d\mu_{\lam})$, then so is
$g^{\lam,W}_{m,k,r}$.

\noindent \textbf{Step 2.} We verify the fact that $\mathcal{G}^{\lam,W}_{m,k}$ is associated with
the kernel $\{\partial_x^m \partial_t^k W_t^{\lam}(x,y)\}_{t>0}$. 
Note that the integral connecting $\mathcal{G}^{\lam,W}_{m,k}$ with the kernel must be understood
in Bochner's sense, and the underlying Banach space is $\mathbb{B} = L^r(t^{(k+|m|\slash 2)r-1}dt)$.
By density arguments it suffices to show that
\begin{equation} \label{assoc}
\Big\langle \big\{ \partial_x^m \partial_t^k W_t^{\lam}f \big\}_{t>0}, H \Big\rangle =
\bigg\langle \int_{\rnp} \{\partial_x^m \partial_t^k W_t^{\lam}(x,y)\}_{t>0} f(y)\, d\mu_{\lam}(y),H
\bigg\rangle
\end{equation}
for every $f \in C_c^{\infty}(\rnp)$ and $H(x,t)=H_1(x)H_2(t)$, where $H_1 \in C_c^{\infty}(\rnp)$,
$H_2 \in C_c^{\infty}(\rp)$ and $\support f \cap \support H_1 = \emptyset$. Here 
$\langle \cdot, \cdot \rangle$ means the standard Banach space pairing between 
$L^2_{\mathbb{B}}(d\mu_{\lam})$ and its dual $L^2_{\mathbb{B}^*}(d\mu_{\lam})$, with
$\mathbb{B}^* = L^{r'}(t^{(k+|m|\slash 2)r-1}dt)$.

We start by considering the left-hand side of \eqref{assoc},
\begin{align*}
& \Big\langle \big\{ \partial_x^m \partial_t^k W_t^{\lam}f \big\}_{t>0}, H \Big\rangle \\
& = \int_0^{\infty} t^{(k+|m|\slash 2)r-1} H_2(t) \int_{\rnp} \partial_x^m \partial_t^k
	W_t^{\lam}f(x) H_1(x)\, d\mu_{\lam}(x)\, dt \\
& = \int_0^{\infty} t^{(k+|m|\slash 2)r-1} H_2(t) \int_{\rnp}\int_{\rnp} \partial_x^m \partial_t^k
	W_t^{\lam}(x,y)f(y)\, d\mu_{\lam}(y) \, H_1(x)\, d\mu_{\lam}(x)\, dt.
\end{align*}
The second identity above follows from Lemma \ref{lem:heat}.
The change of the order of integration in the first identity is justified by Fubini's theorem.
Its application is indeed legitimate since by H\"older's inequality
\begin{align*}
& \int_{\rnp} \int_0^{\infty} \big| \partial_x^m \partial_t^k W_t^{\lam} f(x) \big| \, |H_1(x)H_2(t)|
	t^{(k+|m|\slash 2)r-1} \, dt \, d\mu_{\lam}(x)\\
& \le \big\| g^{\lam,W}_{m,k,2}(f)\big\|_{L^2(d\mu_{\lam})} \|H_1\|_{L^2(d\mu_{\lam})}
	\big\| t^{(k+|m|\slash 2)(r-1)-1\slash 2}H_2\big\|_{L^2(dt)},
\end{align*}
and this quantity is finite in view of the $L^2$-boundedness of $g^{\lam,W}_{m,k,2}$.

Now we focus on the right-hand side of \eqref{assoc}.
Interchanging the order of integrals we get
\begin{align*}
& \bigg\langle \int_{\rnp} \{\partial_x^m \partial_t^k W_t^{\lam}(x,y)\}_{t>0} f(y)\, d\mu_{\lam}(y),H
\bigg\rangle \\
& = \int_0^{\infty} t^{(k+|m|\slash 2)r-1} H_2(t) \int_{\rnp}\int_{\rnp} \partial_x^m \partial_t^k
	W_t^{\lam}(x,y)f(y)\, d\mu_{\lam}(y) \, H_1(x)\, d\mu_{\lam}(x)\, dt,
\end{align*}
which confirms that both sides of \eqref{assoc} coincide. Here application of Fubini's theorem is
possible since
\begin{align*}
& \int_{\rnp}\int_{\rnp} \int_0^{\infty} t^{(k+|m|\slash 2)r-1}  \big| \partial_x^m \partial_t^k
	W_t^{\lam}(x,y)f(y) H_1(x)H_2(t)\big|\, dt \, d\mu_{\lam}(y) \, d\mu_{\lam}(x) \\
& \le \|f\|_{\infty} \|H_1\|_{\infty} \|H_2\|_{\mathbb{B}^*}
	\int_{\support H_1} \int_{\support f} \big\| \partial_x^m \partial_t^k W_t^{\lam}(x,y)
		\big\|_{\mathbb{B}} \, d\mu_{\lam}(y) \, d\mu_{\lam}(x) \\
& \lesssim \int_{\support H_1} \int_{\support f} \frac{1}{\mu_{\lam}(B(x,|x-y|))}
	 \, d\mu_{\lam}(y) \, d\mu_{\lam}(x) < \infty,
\end{align*}
where we made use of the growth estimate for the kernel 
$\{\partial_x^m \partial_t^k W_t^{\lam}(x,y)\}_{t>0}$ proved in Step 3 below and the fact that
the supports of $f$ and $H_1$ are disjoint and bounded.

\noindent \textbf{Step 3.} We verify the standard estimates for the kernel 
$\{\partial_x^m \partial_t^k W_t^{\lam}(x,y)\}_{t>0}$ taking values in 
$\mathbb{B}=L^r(t^{(k+|m|\slash 2)r-1}dt)$, where $1\le r < \infty$.
By \eqref{heatform} and Lemma \ref{lem:estexp} we have
$$
\big| \partial_x^m \partial_t^k W_t^{\lam}(x,y)\big| \lesssim 
	\frac{1}{t^{n\slash 2 +|\lam| + k + |m|\slash 2}} \int \exp\Big(-\frac{\q}{8t}\Big)\, d\Omega_{\lam}(s).
$$
This combined with Minkowski's integral inequality gives
$$
\big\| \partial_x^m \partial_t^k W_t^{\lam}(x,y)\big\|_{\mathbb{B}} \lesssim \int \Big\|
	\frac{1}{t^{n\slash 2 +|\lam| + k + |m|\slash 2}}  \exp\Big(-\frac{\q}{8t}\Big)
	\Big\|_{\mathbb{B}}\, d\Omega_{\lam}(s).
$$
The change of variable $t \mapsto \q t$ shows that the norm under the last integral is equal to
$\q^{-n\slash 2-|\lam|}$, up to a multiplicative constant. Thus now the growth estimate
\eqref{gr} follows from Lemma \ref{lem:bridge}.

To show the smoothness estimates, we focus ourselves only on \eqref{sm2}; the reasoning proving
\eqref{sm1} is essentially the same. By the Mean Value Theorem, 
$$
\big| \partial_x^m \partial_t^k W_t^{\lam}(x,y) - \partial_x^m \partial_t^k W_t^{\lam}(x,y') \big| \le
	|y-y'| \Big| \nabla_{\! y} \partial_x^m \partial_t^k W_t^{\lam}(x,y)\big|_{y=\theta}\Big|,
$$
where $\theta$ is a convex combination of $y$ and $y'$ depending also on $t$.
Thus is suffices to show the estimates
$$
\Big\| \partial_{y_i}\partial_x^m \partial_t^k W_t^{\lam}(x,y)\big|_{y=\theta}\Big\|_{\mathbb{B}}
	\lesssim \frac{1}{|x-y|\mu(B(x,|x-y|))}, \qquad i=1,\ldots,n,
$$
for all $x,y,y'$ satisfying $|x-y|>2|y-y'|$. Using \eqref{heatform} and Lemma \ref{lem:estexp}
we see that
$$
\big|\partial_{y_i} \partial_x^m \partial_t^k W_t^{\lam}(x,y)\big| \lesssim 
	\frac{1}{t^{n\slash 2 +|\lam| + k + |m|\slash 2+1\slash 2}} 
		\int \exp\Big(-\frac{\q}{8t}\Big)\, d\Omega_{\lam}(s).
$$
This together with Lemma \ref{lem:theta} produces (under assumption $|x-y|>2|y-y'|$)
$$
\Big|\partial_{y_i} \partial_x^m \partial_t^k W_t^{\lam}(x,y)\big|_{y=\theta}\Big| \lesssim 
	\frac{1}{t^{n\slash 2 +|\lam| + k + |m|\slash 2+1\slash 2}} 
		\int \exp\Big(-\frac{\q}{32t}\Big)\, d\Omega_{\lam}(s).
$$
From here we proceed as in the proof of the growth estimate above. This leads to
$$
\Big\| \partial_{y_i}\partial_x^m \partial_t^k W_t^{\lam}(x,y)\big|_{y=\theta}\Big\|_{\mathbb{B}}
	\lesssim \int \frac{1}{\q^{n\slash 2+ |\lam|+1\slash 2}} \, d\Omega_{\lam}(s),
$$
and the desired bound follows by Lemma \ref{lem:bridge}. Step 3 is complete.

\subsection{Multipliers of Laplace transform type $T_{\mathcal{M}}^{\lam}$} \label{ssec:Laplace} \quad\\
We will consider multipliers $\mathcal{M}$ of the form
$$
\mathcal{M}(z) = \mathcal{M}_{W}(z) =
|z|^2\int_0^{\infty} e^{-|z|^2s}\psi(s)\, ds, \qquad z \in \rnp,
$$
where $\psi$ is a bounded function on $\rp$.
The arguments given below and the subordination principle allow to treat in a similar
way the multipliers $\mathcal{M}=\mathcal{M}_P$
obtained from the formula above by replacing $|z|^2$ by $|z|$ (this change corresponds
to replacing the heat kernel by the Poisson kernel in the expression defining the 
associated integral kernel).

\noindent \textbf{Step 1.} The fact that the operator 
$T^{\lam}_{\mathcal{M}}f = h_{\lam}(\mathcal{M} h_{\lam}f)$ is bounded on $L^2(d\mu_{\lam})$ is
clear, in view of the boundedness of $\mathcal{M}$ and Plancherel's theorem for the Hankel transform.

\noindent \textbf{Step 2.} We will show that $T^{\lam}_{\mathcal{M}}$ is associated with the
kernel
$$
K^{\lam}_{\mathcal{M}}(x,y) = - \int_0^{\infty} \psi(t) \, \partial_t W_t^{\lam}(x,y)\, dt, \qquad
	x,y \in \rnp, \quad x\neq y.
$$
By density arguments, it suffices to verify that
\begin{equation} \label{assocM}
\big\langle T_{\mathcal{M}}^{\lam}f,g\big\rangle_{d\mu_{\lam}} = \int_{\rnp}\int_{\rnp}
K^{\lam}_{\mathcal{M}}(x,y) f(y)\overline{g(x)}\, d\mu_{\lam}(y)\, d\mu_{\lam}(x)
\end{equation}
for all $f,g \in C_c^{\infty}$ with disjoint supports.
We focus on the left-hand side of \eqref{assocM}. By the definition of $T^{\lam}_{\mathcal{M}}$
and Plancherel's theorem for the Hankel transform we get
\begin{align} \nonumber
\big\langle T_{\mathcal{M}}^{\lam}f,g\big\rangle_{d\mu_{\lam}} & =
\int_{\rnp} h_{\lam}\big(\mathcal{M}h_{\lam}f\big)(z) \overline{g(z)}\, d\mu_{\lam}(z) \\
& = \int_{\rnp} \mathcal{M}(z) h_{\lam}f(z) h_{\lam}\overline{g}(z)\, d\mu_{\lam}(z) \nonumber \\
& = \int_0^{\infty} \psi(t) \int_{\rnp} |z|^2 e^{-t|z|^2} h_{\lam}f(z)
	h_{\lam}\overline{g}(z)\, d\mu_{\lam}(z)\, dt. \label{inner}
\end{align}
To write the last equality we used Fubini's theorem. Its application is justified since
\begin{align*}
\int_{\rnp} \int_0^{\infty} \big| \psi(t)\, |z|^2 e^{-t|z|^2} 
	h_{\lam}f(z) h_{\lam}\overline{g}(z)\big| \, d\mu_{\lam}(z)\, dt
	& \le \|\psi\|_{\infty} \int_{\rnp} \big| h_{\lam}f(z) h_{\lam}\overline{g}(z)\big|\, d\mu_{\lam}(z)\\
& \le \|\psi\|_{\infty} \|f\|_{L^2(d\mu_{\lam})} \|g\|_{L^2(d\mu_{\lam})} < \infty.
\end{align*}
We next analyse the inner integral in \eqref{inner}. Plugging in the integrals defining $h_{\lam}f$
and $h_{\lam}\overline{g}$ and then using Fubini's theorem we arrive at
\begin{align*}
& \int_{\rnp} |z|^2 e^{-t|z|^2} h_{\lam}f(z) h_{\lam}\overline{g}(z)\, d\mu_{\lam}(z)\\ 
& = \int_{\rnp} \int_{\rnp} \int_{\rnp} |z|^2 e^{-t|z|^2} \varphi_z^{\lam}(x) \varphi_z^{\lam}(y)\,
	d\mu_{\lam}(z) \, f(y)\overline{g(x)} \, d\mu_{\lam}(y)\, d\mu_{\lam}(x) \\
& = \int_{\rnp}\int_{\rnp} - \partial_t W_t^{\lam}(x,y) f(y) 
	\overline{g(x)} \, d\mu_{\lam}(y) d\mu_{\lam}(x).
\end{align*}
The application of Fubini in the first identity above is legitimate since, taking into account
\eqref{phiest} and that $t>0$,
\begin{align*}
& \int_{\rnp}\int_{\rnp}\int_{\rnp} \Big| |z|^2 e^{-t|z|^2} \varphi_z^{\lam}(x) \varphi_z^{\lam}(y)
	f(y)\overline{g(x)} \Big| \,d\mu_{\lam}(z) \,d\mu_{\lam}(y) \,d\mu_{\lam}(x) \\
& \lesssim  \|f\|_{L^1(d\mu_{\lam})} \, \|g\|_{L^1(d\mu_{\lam})}
	\int_{\rnp} |z|^2 e^{-t|z|^2} \, d\mu_{\lam}(z) < \infty.
\end{align*}
The second identity above is a consequence of the equality
$$
- \partial_t \int_{\rnp} e^{-t|z|^2} \varphi_z^{\lam}(x) \varphi_z^{\lam}(y)\, d\mu_{\lam}(z)
	= \int_{\rnp} |z|^2 e^{-t|z|^2} \varphi_z^{\lam}(x) \varphi_z^{\lam}(y) \, d\mu_{\lam}(z), 
\qquad t>0.
$$
Here passing with $\partial_t$ under the integral can be easily justified by \eqref{phiest} and the
fact that for $t>\varepsilon>0$ we have $|z|^2 e^{-t|z|^2} < |z|^2 e^{-\varepsilon |z|^2}$, $z\in \rnp$.

Summing up, we proved that 
$$
\big\langle T_{\mathcal{M}}^{\lam}f,g\big\rangle_{d\mu_{\lam}} =
- \int_0^{\infty} \psi(t) \int_{\rnp}\int_{\rnp} \partial_t W_t^{\lam}(x,y) f(y) 
	\overline{g(x)} \, d\mu_{\lam}(y) d\mu_{\lam}(x)\, dt.
$$
To see that this expression coincides with the right-hand side in \eqref{assocM} it is now enough
to interchange the order of integrals. An application of Fubini's theorem is again possible because
\begin{align*}
& \int_{\rnp}\int_{\rnp} \int_0^{\infty} \big| \psi(t) \partial_t W_t^{\lam}(x,y) f(y) 
	\overline{g(x)}\big| \, dt \, d\mu_{\lam}(y) \, d\mu_{\lam}(x) \\
& \lesssim \|f\|_{\infty} \|g\|_{\infty} \int_{\support f} \int_{\support g}
	\frac{1}{\mu_{\lam}(B(x,|x-y|))} \, d\mu_{\lam}(y)\, d\mu_{\lam}(x) < \infty,
\end{align*}
where we used the assumptions on $f$ and $g$ and the estimate
$$
\int_0^{\infty} \big| \psi(t) \partial_t W_t^{\lam}(x,y)\big|\, dt \lesssim 
	\frac{1}{\mu_{\lam}(B(x,|x-y|))}
$$
obtained implicitly in Step 3 below. The verification of \eqref{assocM} is finished.

\noindent \textbf{Step 3.} We show the standard estimates for the (scalar-valued) kernel 
$K^{\lam}_{\mathcal{M}}(x,y)$. By \eqref{heatform} and Lemma \ref{lem:estexp} we have
$$
|\partial_t W_t^{\lam}(x,y)| \lesssim \frac{1}{t^{n\slash 2+|\lam|+1}} \int
	\exp\Big( -\frac{\q}{8t}\Big) \, d\Omega_{\lam}(s).
$$
Using the boundedness of $\psi$ and changing the variable of integration $t \mapsto \q t$ we get
$$
|K^{\lam}_{\mathcal{M}}(x,y)| \lesssim \int \int_0^{\infty} \frac{1}{t^{n\slash 2+|\lam|+1}}
	\exp\Big( -\frac{\q}{8t}\Big)\, dt \, d\Omega_{\lam}(s) \lesssim 
	\int \frac{1}{\q^{n\slash 2+|\lam|}} \, d\Omega_{\lam}(s).
$$
Now the growth estimate \eqref{gr} with $\mathbb{B}=\mathbb{C}$ follows by Lemma \ref{lem:bridge}.

To prove the gradient estimate \eqref{grad}, by symmetry reasons we may consider only the derivatives
$\partial_{x_i}$, $i=1,\ldots,n$. Then Lemma \ref{lem:estexp} gives
$$
|\partial_{x_i}\partial_t W_t^{\lam}(x,y)| \lesssim \frac{1}{t^{n\slash 2+|\lam|+3\slash 2}} \int
	\exp\Big( -\frac{\q}{8t}\Big) \, d\Omega_{\lam}(s)
$$
and hence, proceeding as above,
$$
|\partial_{x_i} K^{\lam}_{\mathcal{M}}(x,y)| \lesssim \int \int_0^{\infty} 
	\frac{1}{t^{n\slash 2+|\lam|+3\slash 2}}
	\exp\Big( -\frac{\q}{8t}\Big)\, dt \, d\Omega_{\lam}(s) 
	\lesssim \int \frac{1}{\q^{n\slash 2+|\lam|+ 1\slash 2}} \, d\Omega_{\lam}(s).
$$
The conclusion follows by Lemma \ref{lem:bridge}.

\subsection{Riesz transforms $R^{\lam}_m$}  \label{sec:riesz}\quad\\
Recall that, see Section \ref{sec:prel}, we defined 
\begin{equation} \label{defR}
R_m^{\lam}f = \partial^m \Delta_{\lam}^{-|m|\slash 2}f =
	\partial^m h_{\lam} \big( |z|^{-|m|} h_{\lam}f\big), \qquad f \in C^{\lam}.
\end{equation}
First of all, we ensure that this definition is correct and that $C^{\lam}$ is dense in $L^2(d\mu_{\lam})$.
Indeed, given $f\in C^{\lam}$, $h_{\lam}f$, and thus also $|z|^{-|m|}h_{\lam}f$, belong to 
$C_c^{\infty}(\rnp)$, by the definition of $C^{\lam}$. Then by Lemma \ref{lem:hsmooth} the function
$h_{\lam}(|z|^{-|m|}h_{\lam}f)$ is smooth, so the formula defining $R_m^{\lam}$ makes sense.
For the density of $C^{\lam}$, observe that $h_{\lam}(C_c^{\infty}(\rnp))\subset C^{\lam}$; this
follows from the definition of $C^{\lam}$, Lemma \ref{lem:hsmooth} and the fact that $h_{\lam}$
coincides with its inverse in $L^2(d\mu_{\lam})$. Since $C_c^{\infty}(\rnp)$ is dense in 
$L^2(d\mu_{\lam})$ and $h_{\lam}$ is an isometry there, the conclusion follows.

\noindent \textbf{Step 1.} We verify that $R_m^{\lam}$ is bounded from $C^{\lam}\subset L^2(d\mu_{\lam})$
to $L^2(d\mu_{\lam})$. In consequence, it extends uniquely to a bounded linear operator on 
$L^2(d\mu_{\lam})$.

Let $f\in C^{\lam}$. Using Lemma \ref{lem:hsmooth} and Lemma \ref{decomp_phi} we can write
\begin{align} \nonumber
R_m^{\lam}f(x) & = \int_{\rnp} \partial_x^m \varphi_x^{\lam}(z) |z|^{-|m|} h_{\lam}f(z)\, d\mu_{\lam}(z) \\
& = \sum_{j \le m} c_j x^j \int_{\rnp} \varphi_x^{\lam+j}(z) |z|^{-|m|} z^{m+j} h_{\lam}f(z)\,
	d\mu_{\lam}(z) \label{strR} \\
& = \sum_{j \le m} c_j x^j h_{\lam+j}\big(|z|^{-|m|} z^{m-j} h_{\lam}f\big)(x). \nonumber	
\end{align}
Then, taking into account that the Hankel transform is an $L^2$-isometry, we see that
\begin{align*}
\big\| R_m^{\lam}f \big\|^2_{L^2(d\mu_{\lam})} & \lesssim \sum_{j\le m} \big\| h_{\lam+j}\big(
	|z|^{-|m|} z^{m-j} h_{\lam}f\big) \big\|^2_{L^2(d\mu_{\lam+j})} \\
& = \sum_{j \le m} \big\| |z|^{-|m|} z^{m} h_{\lam}f\big\|^2_{L^2(d\mu_{\lam})}
	\lesssim \|h_{\lam}f\|^2_{L^2(d\mu_{\lam})} = \|f\|^2_{L^2(d\mu_{\lam})}.
\end{align*}

\noindent \textbf{Step 2.} We show that $R_m^{\lam}$ is associated with the kernel
$$
R_m^{\lam}(x,y) = \frac{1}{\Gamma(|m|\slash 2)} \int_0^{\infty} \partial_x^m W_t^{\lam}(x,y)
	\, t^{|m|\slash 2-1}\, dt, \qquad x,y\in \rnp, \quad x\neq y.
$$
By density arguments, it is enough to justify the identity
\begin{equation} \label{assocR}
\big\langle R_m^{\lam}f,g\big\rangle_{d\mu_{\lam}} = \int_{\rnp}\int_{\rnp} R_m^{\lam}(x,y)
f(y) \overline{g(x)}\, d\mu_{\lam}(y)\, d\mu_{\lam}(x)
\end{equation}
for all $f,g \in C_c^{\infty}(\rnp)$ such that $\support f \cap \support g = \emptyset$.

First we focus on the left-hand side of \eqref{assocR}. Notice that here we cannot apply \eqref{defR}
directly to $R_m^{\lam}f$ because $f$ is not necessarily in $C^{\lam}$. To overcome this obstacle,
we shall use a limiting argument to express $\langle R_m^{\lam}f,g\rangle_{d\mu_{\lam}}$ by a sum
of certain integrals. Let $\{f_N\}\subset C^{\lam}$ be a sequence approximating $f$ in $L^2(d\mu_{\lam})$.
Clearly, in view of the $L^2$-boundedness of $R_m^{\lam}$,
$\langle R_m^{\lam}f_N,g\rangle_{d\mu_{\lam}} \to \langle R_m^{\lam}f,g\rangle_{d\mu_{\lam}}$ as
$N \to \infty$. On the other hand, for each $N$ we can write (see \eqref{strR})
$$
\big\langle R_m^{\lam}f_N,g\big\rangle_{d\mu_{\lam}} = \sum_{j \le m} c_j \int_{\rnp} \int_{\rnp}
(xz)^j \varphi_x^{\lam+j}(z) z^{m} |z|^{-|m|} h_{\lam}f_N(z)\, d\mu_{\lam}(z)\, \overline{g(x)}\,
	d\mu_{\lam}(x).
$$
Then interchanging the order of integrals we get
\begin{equation} \label{RL2}
\big\langle R_m^{\lam}f_N,g\big\rangle_{d\mu_{\lam}} = \sum_{j \le m} c_j \int_{\rnp} z^{m+j} |z|^{-|m|}
	h_{\lam}f_N(z) h_{\lam+j}\big(x^{-j}\overline{g}\big)(z)\, d\mu_{\lam}(z).
\end{equation}
Here the application of Fubini's theorem is legitimate since
\begin{align*}
& \int_{\rnp} \big| z^{m+j} |z|^{-|m|} h_{\lam}f_N(z) h_{\lam+j}\big(x^{-j}\overline{g}\big)(z)\big|
	\, d\mu_{\lam}(z)\\  & \quad \le
  \big\| |z|^{-|m|} z^m h_{\lam}f_N \big\|_{L^2(d\mu_{\lam})} \big\| h_{\lam+j}(x^{-j}\overline{g})
	\big\|_{L^2(d\mu_{\lam+j})}
 \le \|f_N\|_{L^2(d\mu_{\lam})} \, \|g\|_{L^2(d\mu_{\lam})} < \infty.
\end{align*} 
This chain of estimates shows also that one can pass to the limit with $N$ in \eqref{RL2}. Thus we get
$$
\big\langle R_m^{\lam}f,g\big\rangle_{d\mu_{\lam}} = \sum_{j \le m} c_j \int_{\rnp} z^{m+j} |z|^{-|m|}
	h_{\lam}f(z) h_{\lam+j}\big(x^{-j}\overline{g}\big)(z)\, d\mu_{\lam}(z).
$$

We next analyse the right-hand side of \eqref{assocR}. Taking into account the assumptions imposed on
$f$ and $g$ and the estimate
$$
\int_0^{\infty} \big|\partial_x^m W_t^{\lam}(x,y)\big| t^{|m|\slash 2 -1}\, dt \lesssim 
	\frac{1}{\mu_{\lam}(B(x,|x-y|))}
$$
proved in Step 3 below, we may apply Fubini's theorem to get
\begin{align*}
& \int_{\rnp}\int_{\rnp} R_m^{\lam}(x,y) f(y) \overline{g(x)}\, d\mu_{\lam}(y)\, d\mu_{\lam}(x) \\
& = \frac{1}{\Gamma(|m|\slash 2)}\int_0^{\infty} t^{|m|\slash 2-1} \int_{\rnp}\int_{\rnp} \partial_x^m
	\int_{\rnp} e^{-t|z|^2} \varphi_x^{\lam}(z) \varphi_y^{\lam}(z)\, d\mu_{\lam}(z) \, 
		f(y)\overline{g(x)}\, d\mu_{\lam}(y)\, d\mu_{\lam}(x)\, dt.
\end{align*}
We now focus on the three inner integrals entering the last expression. Observe that we may exchange
$\partial_x^m$ with the integral against $d\mu_{\lam}(z)$; this can be justified by means of 
Lemma \ref{decomp_phi}, \eqref{phiest} and the dominated convergence theorem. Then an application
of Lemma \ref{decomp_phi} leads to
\begin{align*}
& \int_{\rnp}\int_{\rnp} \partial_x^m
	\int_{\rnp} e^{-t|z|^2} \varphi_x^{\lam}(z) \varphi_y^{\lam}(z)\, d\mu_{\lam}(z) \, 
		f(y)\overline{g(x)}\, d\mu_{\lam}(y)\, d\mu_{\lam}(x) \\
& = \int_{\rnp}\int_{\rnp}\int_{\rnp} e^{-t|z|^2} z^m \sum_{j\le m} c_j (xz)^j \varphi_x^{\lam+j}(z)
	\varphi_y^{\lam}(z) \, d\mu_{\lam}(z)\, f(y)\overline{g(x)}\, d\mu_{\lam}(y)\, d\mu_{\lam}(x) \\
& = \sum_{j \le m} c_j \int_{\rnp} e^{-t|z|^2} z^{m+j} h_{\lam}f(z) h_{\lam+j}(x^{-j}\overline{g})(z)\,
	d\mu_{\lam}(z).
\end{align*}
The last identity is a consequence of Fubini's theorem, and its application is possible because,
in view of the assumptions imposed on $f$ and $g$, \eqref{phiest} and the fact that $t>0$,
$$
\int_{\rnp}\int_{\rnp}\int_{\rnp} e^{-t|z|^2} z^{m+j} x^j \big| \varphi_x^{\lam+j}(z)\varphi_y^{\lam}(z)
	f(y)\overline{g(x)}\big|\, d\mu_{\lam}(z)\, d\mu_{\lam}(y)\, d\mu_{\lam}(x) < \infty.
$$
Summing up, we proved that the right-hand side of \eqref{assocR} is equal to
$$
\sum_{j \le m} c_j \frac{1}{\Gamma(|m|\slash 2)} \int_0^{\infty} t^{|m|\slash 2-1} \int_{\rnp}
	e^{-t|z|^2} z^{m+j} h_{\lam}f(z) h_{\lam+j}(x^{-j}\overline{g})(z)\, d\mu_{\lam}(z)\, dt.
$$
To finish the proof of \eqref{assocR} it suffices now to justify the possibility of exchanging
the order of integrals in the last expression. This, however, follows readily by means of the estimate
enabling the application of Fubini's theorem leading to \eqref{RL2}.

\noindent \textbf{Step 3.} We prove the standard estimates for the kernel $R_{m}^{\lam}(x,y)$.
By \eqref{heatform} and Lemma \ref{lem:estexp} we have
$$
\big| \partial_x^{m} W_t^{\lam}(x,y)\big| \lesssim \frac{1}{t^{n\slash 2 + |\lam|+|m|\slash 2}}
	\int \exp\Big( -\frac{\q}{8t}\Big) \, d\Omega(s).
$$
This implies
$$
\big| R_m^{\lam}(x,y) \big| \lesssim \int \int_0^{\infty} \frac{1}{t^{n\slash 2 + |\lam|+1}}
	\exp\Big( -\frac{\q}{8t}\Big) \, dt \, d\Omega(s),
$$
and the right-hand side here was already estimated in the required way in Step 3 of Section
\ref{ssec:Laplace} above.

To show the gradient bound, we observe that again by \eqref{heatform} and Lemma \ref{lem:estexp}
$$
\big| \nabla_{\!x,y} \,\partial_x^{m} W_t^{\lam}(x,y)\big| \lesssim 
	\frac{1}{t^{n\slash 2 + |\lam|+|m|\slash 2+1\slash 2}}
		\int \exp\Big( -\frac{\q}{8t}\Big) \, d\Omega(s)
$$
and consequently
$$
\big| \nabla_{\!x,y} R_m^{\lam}(x,y)\big| \lesssim \int \int_0^{\infty} 
	\frac{1}{t^{n\slash 2 + |\lam|+3\slash 2}}
		\exp\Big( -\frac{\q}{8t}\Big) \, dt \, d\Omega(s).
$$
Now the conclusion follows as in Step 3 of Section \ref{ssec:Laplace}.


\end{document}